\documentclass[11pt]{article}
\usepackage{amssymb,amsmath}
\usepackage{float}
\usepackage{color,fullpage}
\usepackage{cite}
\usepackage[title,titletoc]{appendix}
\usepackage{algorithm,algorithmic}
\usepackage{amsthm}
\newcommand{\RR}{\mathbb{R}}

\newcommand{\dom}{{\mathrm{dom}}\,} 

\newcommand{\cX}{{\mathcal{X}}}
\newcommand{\cY}{{\mathcal{Y}}}
\newcommand{\cI}{{\mathcal{I}}}
\newcommand{\cJ}{{\mathcal{J}}}

\newcommand{\cN}{{\mathcal{N}}}
\newcommand{\rint}{{{\rm int}\,}}

\newcommand{\cT}{{\mathcal{T}}}

\newcommand{\cQ}{{\mathcal{Q}}}

\newtheorem{theorem}{Theorem}[section]
\newtheorem{lemma}{Lemma}[section]
\newtheorem{definition}{Definition}[section]
\newtheorem{example}{Example}[section]
\newtheorem{corollary}{Corollary}[section]

\newtheorem{proposition}{Proposition}[section]

\makeatletter

\@addtoreset{equation}{section} \makeatother

\begin{document}

\title{Proximal-Like Incremental Aggregated Gradient Method with Linear Convergence under Bregman Distance Growth Conditions}

\author{Hui Zhang\thanks{
Department of Mathematics, National University of Defense Technology,
Changsha, Hunan 410073, China.  Email: \texttt{h.zhang1984@163.com}}
\and Yu-Hong Dai\thanks{LSEC, ICMSEC,
Chinese Academy of Sciences. Email: \texttt{dyh@lsec.cc.ac.cn}
}
\and Lei Guo\thanks{Corresponding author. School of Business, East China University of Science and Technology, Shanghai 200237, China. Email: \texttt{lguo@ecust.edu.cn}
}
\and Wei Peng\thanks{Department of Mathematics, National University of Defense Technology,
Changsha, Hunan 410073, China.}
}


\date{\today}

\maketitle

\begin{abstract}
We introduce a unified algorithmic framework, called proximal-like incremental aggregated gradient (PLIAG) method, for minimizing the sum of a convex function that consists of additive relatively smooth convex components and a proper lower semi-continuous convex regularization function, over an abstract feasible set whose geometry can be captured by using the domain of a Legendre function. The PLIAG method includes many existing algorithms in the literature as special cases such as the proximal gradient method, the Bregman proximal gradient method (also called NoLips algorithm), the incremental aggregated gradient method, the incremental aggregated proximal method, and the proximal incremental aggregated gradient method. It also includes some novel interesting iteration schemes. First we show the PLIAG method is globally sublinearly convergent without requiring a growth condition, which extends the sublinear convergence result for the proximal gradient algorithm to incremental aggregated type first order methods. Then by embedding a so-called Bregman distance growth condition into a descent-type lemma to construct a special Lyapunov function, we show that the PLIAG method is globally linearly convergent in terms of both function values and Bregman distances to the optimal solution set, provided that the step size is not greater than some positive constant. These convergence results derived in this paper are all established beyond the standard assumptions in the literature (i.e., without requiring the strong convexity and the Lipschitz gradient continuity of the smooth part of the objective). When specialized to many existing algorithms, our results recover or supplement their convergence results under strictly weaker conditions.

\end{abstract}

\newline



\section{Introduction}
In this paper, we consider the following convex minimization problem:
\begin{equation}\label{main}
\min_{x\in \mathcal{Q}}\quad \Phi (x):= F(x)+h(x),
\end{equation}
where $F(x):=\sum_{n=1}^N f_n(x)$ is a convex function whose component functions $f_n:\RR^d\to \RR$ are convex functions that are smooth on the interior of $\mathcal{Q}$, $h:\RR^d\to(-\infty,\infty]$ is a proper, lower semi-continuous (lsc), and  convex regularization function that is possibly non-smooth, and $\mathcal{Q}\subseteq \RR^d$ is a nonempty closed convex subset with a nonempty interior. Many problems in machine learning, signal and image processing, compressed sensing, statistical estimation, communication systems, and distributed optimization can be modelled into this form.

\subsection{Existing algorithmic schemes}\label{subse:algo}
The well-known method to solve \eqref{main} with $\mathcal{Q}=\RR^d$ is the proximal gradient (PG) method:
\begin{equation}\label{PG}
x_{k+1}=\arg\min_{x\in \RR^d}\Big\{h(x)  + \langle \nabla F(x_k),x-x_k\rangle +\frac{1}{2\alpha_k }\|x-x_k\|^2 \Big\},
\end{equation}
where $\alpha_k$ is a positive step size; see, e.g., \cite{Beck-Teb}. In the more general setting of maximal monotone operators, the PG method is also called the forward-backward splitting method \cite{Lions1979Splitting,Passty1979Ergodic}.
The main merit of the PG method lies in exploiting the smooth plus non-smooth structure of the model. However, when the number of component functions, $N$, is very large, which happens in big data models, evaluating the full gradient of $F$ at some point $x$, $\nabla F(x)=\sum_{n=1}^N \nabla f_n(x)$, is costly and even prohibitive. One efficient approach to overcome this difficulty is to approximate $\nabla F$ by using a single component gradient at each iteration in a cyclic or randomized order \cite{Bertsekas2011a,Bertsekas2011b}.


To provide a better approximation of $\nabla F$, the proximal incremental aggregated gradient (PIAG) method was studied in  \cite{Global2016Vanli,Analysis2016Aytekin,A2016Vanli}. By exploiting the additive structure of the $N$ component functions, at each iteration $k\geq 0$ the PIAG method first constructs a vector that aggregates the gradients of all component functions, evaluated at the $(k-\tau_k^n)$-th iteration, i.e.,
$g_k=\sum_{n=1}^N\nabla f_n(x_{k-\tau_k^n}),$
where $\tau_k^n$ are delayed indexes, to approximate the full gradient of $F$ at the current iteration point $x_k$. After obtaining the approximation $g_k$ of $\nabla F(x_k)$, the PIAG method performs a proximal step on the sum of $h(x)$ and $\langle g_k, x-x_k\rangle$ as follows:
\begin{equation*}\label{xk}
x_{k+1}=\arg\min_{x\in \RR^d}\Big\{ h(x)+\langle g_k, x-x_k\rangle  + \frac{1}{2\alpha_k}\|x-x_k\|^2\Big\}.
\end{equation*}
One can see that the PIAG method differs from the PG method mainly at the approximation of $\nabla F(x_k)$, and these two methods are the same when $\tau_k^n\equiv 0$. In this sense, the PIAG method can be viewed as a generalization of the PG method. In addition, when the regularization function $h$ vanishes, the PIAG method reduces to the incremental aggregated gradient (IAG) method proposed in \cite{A2008Blatt}.
Recently, different from the IAG method, Bertsekas in \cite{Bertsekas2015incre}  proposed the incremental aggregated proximal (IAP) method
\begin{equation}\label{subp1}
x_{k+1}=\arg\min_{x\in \RR^d}\Big\{ f_{i_k}(x)+ \Big\langle \sum_{i\neq i_k}\nabla f_i(x_{k-\tau_k^i}),x-x_k \Big \rangle +\frac{1}{2\alpha_k}\|x-x_k\|^2 \Big\}
\end{equation}
where $i_k\in \{1,\ldots,N\}$ is chosen arbitrarily. The iteration scheme \eqref{subp1} follows the spirit of the proximal point algorithm (see \cite{Rockafellar1976mono}) to keep one of the component functions at each iteration.

When the feasible set $\cQ$ is a proper subset in $\RR^n$, by using the indictor function $\delta_\cQ(x)$, which is equal to 0 when $x\in \cQ$ and $\infty$ otherwise, problem \eqref{main} is equivalent to the problem of minimizing $F(x)+\widetilde{h}(x)$, where $\widetilde{h}(x):=h(x)+\delta_\cQ(x)$.
The PG-type methods with the Euclidean distance can be directly applied to this formulation, but with a potential difficulty of computing the proximal step on $\tilde{h}$. An alternative approach of handling the feasible set is to choose an associated Legendre function $w$ and use its Bregman distance $D_w(\cdot,\cdot)$ as a proximity measure; see, e.g., \cite{Bauschke2016A}. Using this Bregman distance to replace the Euclidean distance $\frac{1}{2\alpha}\|x-x_k\|^2$ in \eqref{PG}, we obtain the so-called NoLips algorithm proposed in \cite{Bauschke2016A} (or the BPG method in \cite{Teboull2018A}):
\begin{equation}
x_{k+1}=\arg\min_{x\in \mathcal{Q}}\Big\{h(x)+\langle \nabla F(x_k), x-x_k\rangle+ \frac{1}{\alpha_k}D_w(x,x_k)\Big\}.
\end{equation}
A summary for some popular choices of $w$ can be found in \cite[Example 1]{Bauschke2016A}. As noted in \cite{Bauschke2016A}, the framework of the NoLips algorithm is not new and it follows the line of mirror descent methods (see, e.g.,\cite{Nemirovski1983Yudin,beck2003fast,Tseng2010,Beck2017First}). It should be noted that the foremost aspect of the NoLips algorithm is to weaken the assumption of Lipschitz gradient  continuity by utilizing the geometry of the constraint set $\cQ$. Using the same idea of replacing the Euclidean distance with a non-Euclidean one, a non-quadratic IAG method was recently proposed  in \cite{Bertsekas2015incre}, but leaving the issue of convergence analysis open.

In this paper, we will propose a unified algorithmic framework that recovers all the iteration schemes mentioned above, and give a simple and unified convergence analysis under mild assumptions. This framework is not a simple unification and it also includes some new interesting  algorithms that will be presented in Section \ref{se3}.

\subsection{Related convergence analysis}

To the best of our knowledge, the IAG method is the first incremental gradient method that uses the aggregated delayed gradients to provide a better approximation of the full gradient of $F$ than using a single component gradient. Although it is a first order gradient-type method, both the global convergence and globally linear convergence are not easy to establish. In \cite{A2008Blatt}, the authors only established the global convergence under some restrictive assumptions and the linear convergence for the case where component functions are quadratic, assuming that the delayed indexes $\tau_k^n$ satisfy certain restrictions. Subsequently, the linear convergence result was established for nonquadratic problems and for various forms of the method \cite{mairal2015Maj,Schmidt2017Minimizing}. Recently, by viewing the IAG iteration as a gradient method with errors and using distances of the iterates to the optimal solution as a Lyapunov function, Gurbuzbalaban et al.  in \cite{Gurbuzbalaban2015on} derived the first linear convergence for the IAG method with allowing the use of arbitrary
indexes $\tau_k^n$ less than a given positive constant. As pointed out by Bertsekas in  \cite{Bertsekas2015incre}, their proof idea is elegant and particularly simple, and can be employed to show the globally linear convergence of the IAP method, but not readily extended to handle the constrained or non-smooth composite cases.

Vanli et al. first established the global convergence rate of the PIAG method in \cite{Global2016Vanli}.  In contrast with the technique in \cite{Gurbuzbalaban2015on}, their convergence analysis uses function values to track the evolution of the iterates generated by the PIAG method, and their results improve the convergence rate of the IAG method.  By introducing a key lemma that characterizes the linear convergence of a Lyapunov function, Aytekin et al. showed that the iterates generated by the PIAG method is globally linearly convergent in \cite{Analysis2016Aytekin}. Later on, by combining the results presented in \cite{Global2016Vanli} and \cite{Analysis2016Aytekin}, a stronger linear convergence rate in terms of function values for the PIAG method  was given in \cite{A2016Vanli}.

In order to guarantee first order methods to converge linearly, one of standard assumptions is the strong convexity of the objective function which, however, may be too stringent in practice. Recently, linear convergence of first order methods for non-strongly convex minimization problems was derived under error bound conditions; see the pioneer paper \cite{Bolte2015From} and other followed papers  \cite{Zhang2015The,I2015Linear,Karimil2016linear,zhang2016new}. We wonder whether one can establish linear convergence for the IAP and PIAG methods without the strong convexity, which is a common assumption used in \cite{Gurbuzbalaban2015on,Bertsekas2015incre,Global2016Vanli,Analysis2016Aytekin,A2016Vanli}. This is also one of our motivations for this work.

Another required assumption in the convergence analysis for first order methods is the Lipschitz gradient continuity of the smooth part of the objective. This assumption is also stringent in practice.  A so-called Lipschitz-like/convexity condition was first introduced in \cite{Bauschke2016A} to circumvent the intricate question of Lipschitz continuity of the gradient and was then used under the term ``relative smoothness" in \cite{Lu2016relatively}.
Bolte et al. further developed this idea to analyze the convergence of nonconvex composition minimization problems in \cite{bolte2017first}.
Our last but not least motivation is to extend their theory to our proposed algorithm.
\subsection{Contributions}

Our contributions made in this paper are three-fold, which are summarized as follows:
\begin{itemize}
\item \textbf{Algorithmic aspect.} We propose a unified algorithmic framework that includes all the methods mentioned in Subsection 1.1 as special cases. There are three ingredients in the proposed framework: incremental aggregation (gradient or proximal), regularization function, and Bregman's distance. It is the first time to put them together although none of them is new. From the algorithmic aspect, the main benefit is that we can collect many related algorithmic schemes  in the literature into a unified form. It also includes some novel interesting schemes due to certain combinations of the three ingredients such as the non-Euclidean PIAG and IAP methods presented in Section \ref{se3}.

\item \textbf{Theoretical aspect.} We establish globally sublinear and linear convergence results for the PLIAG method beyond the standard assumptions in the literature (i.e, without requiring the strong convexity and the Lipschitz gradient continuity of $F$). To the best of our knowledge, we establish the first sublinear convergence for incremental aggregated type first order methods with positive delayed parameters, including the IAP and PIAG methods as special cases.
   When specialized to many existing algorithms, our results recover or supplement their convergence results under strictly weaker conditions. In particular, when specialized to the NoLips algorithm proposed in \cite{Bauschke2016A}, a globally linear convergence result is obtained which supplements the convergence results in \cite{Bauschke2016A} where only the sublinear convergence rate was established.
   In \cite{Bertsekas2015incre}, Bertsekas pointed out that the convergence results for the IAP method applied with a convex constraint are not available. We show that the convergence results can be derived as a special case of our generalized results.
    Moreover for the PIAG method, we find that the quadratic growth condition (weaker than the  strong convexity used in the literature) is enough to ensure the linear convergence and further, an improved locally linear convergence rate can be derived under the stronger H\"{o}lderian growth condition.

\item \textbf{Proof strategy.} Our proof strategy for the linear convergence rate relies on two pillars: one is the Lipschitz-like/convexity condition introduced in \cite{Bauschke2016A} and the other is the Bregman distance growth condition proposed in this paper. A certain Lyapunov function is constructed in our proof by embedding the Bregman distance growth condition into a descent-type lemma. This makes our proof technically different from the existing ones. Moreover, this proof technique can be slightly modified to analyze linear convergence of nonconvex and inertial PIAG methods.

\end{itemize}


\subsection{Organization}

The rest of the paper is organized as follows.  In Section \ref{se2}, we list all the useful assumptions. In Section \ref{se3}, we propose the unified algorithmic framework. In Section \ref{se4}, we establish the globally sublinear and linear convergence for our proposed algorithm under some modified assumptions. As an extension, in Section \ref{se5}, we derive an improved locally linear convergence rate for the PIAG method under a H\"{o}lderian growth condition.  Concluding remarks are given in Section
\ref{se6}. All proofs are given in Appendices.

\bigskip

{\bf Notation.} The notation used in this paper is standard as in the literature. $\|\cdot\|$ stands for the
Euclidean norm. For any $x\in \RR^d$ and any nonempty set $\Omega\subseteq \RR^d$, the Euclidean distance from $x$ to $\Omega$ is defined by
$d(x,\Omega):=\inf_{y\in \Omega}\|x-y\|.$ We let $\delta_\Omega(\cdot)$ stand for the indicator function which is equal to 0 if $x\in \Omega$ and $\infty$ otherwise, and $\overline{\Omega}$ and $\rint \Omega$ denote the closure and interior of $\Omega$ respectively. An effective domain of an extended function $\varphi:\RR^d\to (0,\infty]$ is defined as $\dom \varphi := \{x\in \RR^d: \varphi(x)<\infty\}$.
Let $\cX$ be the optimal solution set of problem \eqref{main} and $\Phi^*$ the optimal objective function value. We assume that $\cX$ is nonempty and compact. If $\cX$ is a singleton,  let $x^*$ denote the unique optimal solution. For simplicity, we let $\cN:=\{1,2,\ldots,N\}$ and $\cT:=\{0,1,\ldots,\tau\}$ for a upper bound $\tau$ for delayed indexes.

\section{Assumptions}\label{se2}

In this section, we give all the assumptions that will be used in this paper.
First of all, let us recall the definition of a Legendre function; see, e.g.,  \cite{Rock1970convex}. Let $\varphi: \RR^d\rightarrow(-\infty, +\infty]$ be a proper lsc convex function. We say that it is essentially smooth if $\rint\dom{\varphi}\neq\emptyset$, $\varphi$ is differentiable on $\rint\dom{\varphi}$, and  $\|\nabla \varphi(x_k)\|\rightarrow \infty$ for every sequence $\{x_k\}_{k\geq 0}\subseteq \rint\dom{\varphi}$ converging to a boundary point of $\dom{\varphi}$ as $k\rightarrow \infty$. We say that it is of Legendre type if $\varphi$ is essentially smooth and strictly convex on $\rint\dom{\varphi}$.

For the constraint set $\cQ$, by e.g., \cite[Exampple 3.6]{legendre}, there must exist a Legendre function $w:\RR^d\to (-\infty,\infty]$ such that $\cQ=\overline{\dom{w}}$. Associated with this Legendre function $w$, the Bregman distance is defined as
$$D_w(y,x)=w(y)-w(x)-\langle \nabla w(x), y-x\rangle,\quad  \forall y\in \dom{w}, x\in \rint\dom{w}.$$
In contrast to the Euclidean distance, the Bregman distance is generally not symmetric, i.e., $D_w(y,x)\neq D_w(x,y)$ if $x\neq y$. If $\cQ=\RR^d$ and $w(x)=\frac{1}{2}\|x\|^2$, then $D_w(y,x)=\frac{1}{2}\|y-x\|^2$. In this sense, the Bregman distance generalizes the Euclidean distance. The following two simple but remarkable properties will be used:
\begin{itemize}
  \item The three-point identity (see \cite{Chen1993conv}). For any points $x\in \dom{w}, y,z\in \rint\dom{w}$, we have
$$D_w(x,z)- D_w(x,y)- D_w(y,z)=\langle \nabla w(y)-\nabla w(z),x- y \rangle.$$
  \item Linear additivity (see \cite{bolte2017first}). For any constants $\alpha,\beta \in \RR$ and any functions $w_1, w_2$, we have
  $$D_{\alpha w_1+\beta w_2}(x,y)=\alpha D_{w_1}(x,y)+\beta D_{w_2}(x,y),$$
  for all $x, y\in \dom{w_1}\bigcap \dom{w_2}$ such that both $w_1$ and $w_2$ are differentiable at $y$.
\end{itemize}

\subsection{Standard assumptions}
The following are the standard assumptions in the literature which are also used in \cite{Analysis2016Aytekin} and \cite{A2016Vanli}.
\begin{itemize}
  \item[A0.] The time-varying delays $\tau_k^n$ are bounded, i.e., there is a nonnegative integer $\tau$ such that
  $$\tau_k^n\in \cT, \quad \forall k\geq 0, n\in \cN. $$
Such $\tau$ is called the smallest upper bound for all delayed indexes.
  \item[A1.] Each component function $f_n$ is convex with $L_n$-Lipschitz continuous gradient, i.e.,
  $$\|\nabla f_n(x)-\nabla f_n(y)\|\leq L_n\|x-y\|, ~~\forall x, y\in \RR^d.$$
  This assumption implies that the sum function $F$ is convex with $L$-Lipschitz continuous gradient, where $L=\sum_{n=1}^N L_n$.
  \item[A2.] The function $h$ is subdifferentiable in its effective domain, i.e., $\partial h(x)\neq \emptyset$ for $x\in \dom{h}$.
  \item[A3.] The sum function $F$ is strongly convex on $\RR^d$, i.e., there exists $\nu>0$ such that the function $F(x)-\frac{\nu}{2}\|x\|^2$ is convex on $\RR^d$.
\end{itemize}


\subsection{Growth conditions}\label{subse:growth}

Recently, many research indicates that the strong convexity assumption is not a necessary condition to obtain linear convergence of first order methods; see, e.g, \cite{Bolte2015From}. It can be replaced by the following quadratic growth condition:
\begin{itemize}
  \item[A3a.] The function $\Phi$ satisfies the quadratic growth condition, i.e., there is $\nu>0$ such that
      \begin{equation}\label{growth}
      \Phi(x)-\Phi^*\geq \frac{\nu}{2}d^2(x,\cX),\quad \forall x\in \cQ.
      \end{equation}
\end{itemize}
This condition may be seen as a generalization of the strong convexity of $\Phi$. Indeed, by the subgradient inequality for the strong convex function $\Phi+\delta_{\cQ}$, one can show that \eqref{growth} holds, and thus assumption A3 is stronger than assumption A3a. However, we can easily construct functions to illustrate that the reverse implication fails. For example, the composition function $g(Ax)$, where $g$ is a strongly convex function and $A$ is a rank deficient matrix, satisfies the quadratic growth condition but it is not strongly convex. In this sense, assumption A3a is strictly weaker than assumption A3. See Appendix A for more examples.

The quadratic growth condition \eqref{growth} is a special case of the following H\"{o}lderian growth condition by setting $\theta=1$. It is easy to see that when $\theta$ is smaller, the H\"{o}lderian growth condition is weaker when $d(x,\cX)\ge1$, while it is stronger when $d(x,\cX)\le1$. Thus, in the latter case, it is reasonable to expect a faster convergence speed when $\theta$ becomes smaller. This has been shown to hold true for the PG method; see \cite{Garrigos2017conv,Bolte2015From}.

\begin{itemize}
  \item[A3b.] The function $\Phi$ satisfies the H\"{o}lderian growth condition with $0< \theta\leq 1$, i.e.,  there is $\mu>0$ such that
      \begin{equation}\label{pgrow}
  \Phi (x)-\Phi^*\geq \frac{\mu}{2}d^{2\theta}(x,\cX),\quad \forall x\in\cQ.
\end{equation}

\end{itemize}
In this paper, we will show that for the PIAG method, faster convergence rate can also be expected when $\theta$ is smaller if  the distance from the initial point to $\cX$ is not greater than 1.

\subsection{Modified assumptions}\label{subsection}

In this subsection, we introduce a group of assumptions that generalize the standard assumptions in the literature. Recall that a suitable function $w$ is predetermined such that $\cQ=\overline{\dom{w}}$. It should be noted that the foremost factor in choosing $w$ is that the subproblems will be computationally tractable, and the strong convexity and Lipschitz gradient continuity of $F$ can be weakened.


\begin{itemize}
  \item[B1.] Each component function $f_n$ is proper lsc convex with  $\dom{w} \subseteq \dom{f_n}$, which is differentiable on $\rint\dom{w}$. In addition, for any $n\in\cN$,
      the pair of functions $(w,f_n)$ satisfies a Lipschitz-like/convexity condition, i.e., there exists $L_n>0$ such that $L_nw-f_n$ is convex on $\rint\dom w$.

\item[B2.] $\dom{h}\cap \rint\dom w\neq \emptyset$. Moreover, $h$ is subdifferentiable in $\dom{h}\cap \rint\dom w$.

\item[B3.]
The function $\Phi$  satisfies the Bregman distance growth condition, i.e.,  there is $\mu>0$ such that
\begin{equation}\label{BDG}
\Phi(y)-\Phi^*\geq \mu\inf_{z\in \cX}D_w(z,y),~~\forall  y\in \rint\dom{w},
\end{equation}
where we recall that $\cX$ is the optimal solution set.

\item[B4.]  If the smallest upper bound $\tau$ for all delayed indexes is positive, then there exists a monotonically increasing function $\ell(\cdot)$ with $\ell(1)=1$ such that for any sequence $\{v_1,v_2,\ldots, v_{k}\}$ belonging to $\rint\dom{w}$, it holds that
  $$D_w(v_{k},v_j)\leq \ell(k-j)\sum_{i=j}^{k-1} D_w(v_{i+1},v_i),\quad \forall j\ge1, k >j+1.$$
\end{itemize}


The Lipschitz-like/convexity condition in assumption B1 was originally introduced in \cite{Bauschke2016A} to weaken the Lipschitz continuity of the gradient of $f_n$. It is equivalent to
$$
f_n(y)\leq f_n(x)+\langle \nabla f_n(x),y-x\rangle +L_n D_w(y,x),~~\forall x, y\in \rint\dom{w}.
$$
Later on, this condition was used in \cite{Lu2016relatively} under the term ``the relative smoothness of $f_n$ with respect to $w$".
Recently, Bolte et al. further developed this idea to analyze the convergence of nonconvex composition minimization problems in \cite{bolte2017first}.



The Bregman distance growth condition in assumption B3 is a new notion. It can be viewed as a generalization of the quadratic growth condition when using the Bregman distance instead of the Euclidean distance.  It should be noted that the Bregman distance growth condition was also independently used in \cite{Gutman2018A} to derive linear convergence of an accelerated Bregman proximal gradient method and used in \cite{Bauschke2019On} to prove linear convergence of a Bregman gradient method for nonconvex minimization problems, while this work was under review. The condition \eqref{BDG} implicitly assume that $\cX \bigcap \dom{w} \neq \emptyset$. In Appendix A, we  give two classes of problems satisfying the quadratic growth condition but the standard assumption A3 fails, and in Appendix B, we propose two sufficient conditions in terms of properties of the Legendre function $w$ for assumption B3 to hold, and give a class of functions satisfying the Bregman distance growth condition.  Finding more practical functions satisfying the Bregman distance growth condition is a very
interesting topic and deserves further study.


Assumption B4 holds automatically with $\ell(\cdot)$ being the identity function if $w(x)=\frac{1}{2}\|x\|^2$. It can be seen as an analog of the following inequality:
$$
\|v_{k}-v_j\|^2=\|\sum_{i=j}^{k-1}(v_{i+1}-v_i)\|^2\leq (k-j)\sum_{i=j}^{k-1}\|v_{i+1}-v_i\|^2,\quad \forall  k>j\ge1.
$$
This assumption is not stronger than the  condition used in \cite{Analysis2016Aytekin}: there exist $\mu_w>0$ and $L_w>0$ such that
\begin{equation*}\label{BC}
\frac{\mu_w}{2}\|x-y\|^2\leq D_w(x,y)\leq \frac{L_w}{2}\|x-y\|^2,
\end{equation*}
holding automatically if $w$ is strongly convex and has a Lipschitz gradient, since for all $k>j+1$,
$$
D_w(v_{k},v_j)\leq \frac{L_w}{2}\|v_{k}-v_j\|^2\leq \frac{(k-j) L_w}{2}\sum_{i=j}^{k-1}\|v_{i+1}-v_i\|^2\leq\frac{(k-j)L_w}{\mu_w}\sum_{i=j}^{k-1} D_w(v_{i+1},v_i).
$$

\section{Unified algorithmic framework}\label{se3}

We describe our unified algorithmic framework for solving problem \eqref{main} as follows.
\bigskip
\\
\fbox{
\begin{minipage}[l]{0.91\textwidth}
\begin{flushleft}
{\bf Proximal-Like Incremental Aggregated Gradient (PLIAG) Method:}
\noindent\item[(i)] Choose a Legendre function $w$ such that $\cQ=\overline{\dom w}$. Choose $\tau\ge 0$ as the smallest upper bound for all delayed indexes and an arbitrary initial point $x_0\in \rint \dom w$. Assume that $x_k=x_0$ for all $k<0$. Let $k=0$. \\

\item[(ii)] Choose $\alpha_k>0$,  $\emptyset\neq\cJ_k\subseteq \cJ_k$, and $\tau_k^n\in\cT$ for all $n\in \cN$. Let $\cI_k$ be the complement of $\cJ_k$ with respect to $\cN$. The next iteration point $x_{k+1}$ is obtained via
      \begin{equation}\label{frame1}
      x_{k+1}=\arg\min_{x\in \mathcal{Q}}\Big\{\Phi_k(x):= h(x)+ \sum_{i\in \cI_k} f_{i}(x)+ \Big\langle \sum_{j\in \cJ_k}\nabla f_j(x_{k-\tau_k^j}),x\Big\rangle +\frac{1}{\alpha_k}D_w(x,x_k) \Big\}.
      \end{equation}
\item[(iii)] Set $k \leftarrow k+1$ and go to Step ii).
\end{flushleft}
\end{minipage}
}

\bigskip


We also remark that the foremost factor in choosing $w$ is that the subproblem \eqref{frame1} will be computationally tractable, and the strong convexity and Lipschitz gradient continuity of $F$ can be weakened. The iterate \eqref{frame1} can be equivalently written as
$
x_{k+1}=\arg\min_{x\in \RR^d} \{\Phi_k(x)\}
$ since $\dom{w} \subseteq \mathcal{Q}$.
The PLIAG method is well-defined  as shown in the following result.
\begin{proposition}\label{pro1}
Assume that assumptions B1 and B4 hold. Let $\alpha_k\leq \frac{1}{\sum_{j\in J_k}L_j}$ when $\tau=0$ and $\alpha_k\leq \frac{1}{\ell(2)\sum_{j\in J_k}L_j}$ when $\tau>0$. If $x_j\in \rint\dom{w}$ for all $j\leq k$, then the problem in \eqref{frame1} has a unique optimal solution which belongs to $\rint\dom w$.
\end{proposition}
The PLIAG method is a general algorithmic framework that includes all the methods mentioned in Subsection \ref{subse:algo}  as special cases. In particular, when $\cQ = \RR^d$, $w(x)=\frac{1}{2}\|x\|^2$, and $D_w(x,x_k)=\frac{1}{2}\|x-x_k\|^2$, one can have
\begin{itemize}
  \item[a).]  if $\cI_k\equiv\emptyset$ and $h(x)\equiv0$, then the iterate \eqref{frame1} recovers the IAG method;
  \item[b).]  if $\cI_k\equiv\{i_k\}$ and $h(x)\equiv0$, then the iterate \eqref{frame1} recovers the IAP method;
  \item[c).]  if $\cI_k\equiv\emptyset $, then the iterate \eqref{frame1} recovers the PIAG method.
\end{itemize}
Moreover, the PLIAG method includes two novel interesting methods as follows.

{\bf Non-Euclidean PIAG method.}
Letting $\cI_k\equiv\emptyset $, the iterate \eqref{frame1} becomes
\begin{equation}\label{PIAG-B}
x_{k+1}=\arg\min_{x\in \RR^d}\Big\{h(x)+\Big\langle \sum_{n=1}^N\nabla f_n(x_{k-\tau_k^n}), x \Big\rangle + \frac{1}{\alpha_k}D_w(x,x_k)\Big\}.
\end{equation}
The iterate \eqref{PIAG-B} can be viewed as an incremental version of the NoLips algorithm recently studied in \cite{Bauschke2016A} (i.e., the iterate \eqref{PIAG-B} with $N=1$ and $\tau_k^n\equiv 0$). The computation complexity in each iteration of the two algorithms is the same once the full gradient $\sum_{n=1}^N\nabla f_n(x_{k})$ and the approximate gradient $\sum_{n=1}^N\nabla f_n(x_{k-\tau_k^n})$ are given. However, when the number $N$ is large, the computation of the full gradient is costly and even prohibitive, while the computation cost of the approximate gradient may be very low since it uses historical gradients of component functions. In \cite{Bauschke2016A}, the NoLips algorithm is applied  to solve the Poisson linear inverse problem, whose smooth part consists of many component functions. As the number of component functions is large, the iteration scheme \eqref{PIAG-B} may be a more reasonable choice to solve this problem.
While the iterate \eqref{PIAG-B} is structurally the same as the one in \cite{Analysis2016Aytekin}, as will be seen in Section \ref{se4}, the strong convexity of functions $w$ and $F$, and the Lipschitz gradient continuity of $F$ are not needed in our convergence analysis.

{\bf Non-Euclidean IAP method.}
Letting $\cI_k=\{i_k\}$ and $h(x)\equiv0$, the iterate \eqref{frame1} becomes
\begin{equation}\label{IAP-B}
x_{k+1}=\arg\min_{x\in \RR^d}\Big\{  f_{i_k}(x)+ \Big\langle \sum_{j\neq i_k}\nabla f_j(x_{k-\tau_k^j}),x \Big\rangle +\frac{1}{\alpha_k}D_w(x,x_k) \Big\}.
\end{equation}
This iterate \eqref{IAP-B} can be called a non-Euclidean IAP method, which includes the non-quadratic IAP method in \cite{Bertsekas2015incre} as a special case. The convergence results for the  IAP and non-quadratic IAP methods applied with a convex constraint set are not available in \cite{Bertsekas2015incre} and are leaved open. By choosing a Legendre function $w$ such that $\cQ=\overline{\dom{w}}$, the non-Euclidean IAP method applies to solve convex minimization problems with a convex constraint set beyond the standard assumptions; see Theorems \ref{mainresult0} and \ref{mainresult} in Section \ref{se4}. We point out that the constrained IAP scheme \eqref{IAP-C} introduced in Section \ref{se4} can also apply to solve convex minimization problems with a convex constraint set but its convergence is established under the standard assumptions A1 and A3a (see Corollary \ref{cor4.2}).






As illustrated in the non-Euclidean PIAG method, when setting $\cI_k\equiv\emptyset$, the computation complexity in each iteration of the PLIAG method is the same as that of the NoLips algorithm in \cite{Bauschke2016A}. Including $\cI_k$ in our proposed framework generally makes the computation of the subproblems difficult. The reason why we include $\cI_k$ in the proposed framework is that: (a) the IAP method can be recovered and the derived convergence results can apply to the IAP method, and (b) for some cases, keeping some component functions in each iteration will provide a better approximation for the original problem while might preserve the low computation complexity at each iteration. To illustrate this point, let us consider the augmented sparse minimization problem (also called elastic net model \cite{zou2005Re}):
\begin{equation}\label{elas-net}
\min \Big\{\sum_{i=1}^m \Big\{\langle a_i,x\rangle -b_i\log\langle a_i,x\rangle \Big\}+\beta\|x\|^2+\mu\|x\|_1: x\ge0 \Big\},
\end{equation}
where $0\not=a_i\in \RR^n_+$, $b_i>0$, $\mu, \beta\geq0$ are given. This model with $\beta=0$ was used in \cite{Bauschke2016A} to illustrate how the NoLips algorithm can be applied. Let $g(x):= \sum_{i=1}^m \{\langle a_i,x\rangle -b_i\log\langle a_i,x\rangle \}$ and take Burg's entropy function $w(x)=-\sum_{j=1}^n \log x_j$ with $\dom w=\RR^n_{++}$ as a Legendre function. It has been shown that the pair of functions $(w,g)$ enjoys the Lipschitz-like/convexity condition (see \cite[Lemma 7]{Bauschke2016A}), i.e., $\tilde{L}w-g$ is convex on $\RR^n_{++}$  for any $\tilde{L}\geq L:=\sum_{i=1}^m b_i$. Taking the step size $\alpha=1/L$ and applying the NoLips algorithm to \eqref{elas-net} with $\beta=0$, we arrive at the subproblem
$$x^+=\arg\min\Big\{ \mu\|u\|_1 +\langle \nabla g(x), u\rangle + \frac{1}{\alpha}\sum_{j=1}^n\left(\frac{u_j}{x_j}-\log\frac{u_j}{x_j}-1\right): u>0\Big\}.$$
For the general case when $\beta\neq 0$, letting $\tilde{g}(x):=g(x)+\beta \|x\|^2$ as the smooth part
and $\tilde{w}(x):=w(x)+\frac{\beta}{L}\|x\|^2$ as the associated Legendre function, it follows that $L\tilde{w}-\tilde{g}=Lw-g$ is convex on $\RR^n_{++}$ and hence the pair of functions $(\tilde{w},\tilde{g})$ also enjoys the Lipschitz-like/convexity condition. Then applying the NoLips algorithm  to \eqref{elas-net} yields the iterate as follows:
$$
x^+=\arg\min\Big\{ \mu\|u\|_1 +\langle \nabla g(x), u\rangle + 2\beta\langle x, u\rangle + \frac{1}{\alpha}\sum_{j=1}^n\left(\frac{u_j}{x_j}-\log\frac{u_j}{x_j}-1\right)+\beta \|u-x\|^2: u>0\Big\},
$$
which can be solved analytically. On the other hand, when keeping the term $\mu \|x\|^2$ in each iteration, the corresponding iterate becomes
\begin{equation*}\label{sub}
x^+=\arg\min\Big\{ \mu\|u\|_1 +\langle \nabla g(x), u\rangle + \beta\|u\|^2 + \frac{1}{\lambda}\sum_{j=1}^n\left(\frac{u_j}{x_j}-\log\frac{u_j}{x_j}-1\right)+ \beta \|u-x\|^2: u>0\Big\},
\end{equation*}
which reduces to solve $n$ number of one-dimensional convex problems in the form
$$
 x_i^+=\arg\min\Big\{\mu u_i +\gamma u_i +\beta u_i^2+\frac{1}{\lambda}\left(\frac{u_i}{x_i}-\log \frac{u_i}{x_i}\right)+\beta (u_i-x_i)^2: u_i>0\Big\}\quad  i=1,\ldots,n,
$$
where $\gamma$ stands for the $i$-th component of $\nabla g(x)$. Direct calculation yields that
$$
x_i^+=\frac{\sqrt{(\lambda\mu x_i + \lambda\gamma x_i + 1- 2\lambda\beta x_i^2)^2+ 16\lambda \beta x_i^2}-(\lambda\mu x_i + \lambda\gamma x_i + 1- 2\lambda\beta x_i^2)}{8\lambda\beta x_i}\quad i=1,\ldots,n.
$$

\section{Key Lemmas and Main Results}\label{se4}

In this section, we investigate the convergence behavior of the PLIAG method and give the required key lemmas. To simplify the convergence analysis, we assume that the sequence $\{x_k\}$ is generated by the PLIAG method with the same step size (i.e., $\alpha_k\equiv\alpha$). All the results presented below have evident valid counterparts for the PLIAG method with different step size $\alpha_k$. Throughout this section, we assume that assumption A0 always holds.

The first key lemma can be viewed as a generalization of the standard descent lemma for the PG method  (\cite[Lemma 2.3]{beck2009fast}) and the newly discovered descent lemma for the NoLips method  (\cite[Lemma 5]{Bauschke2016A}).

\begin{lemma}[Descent lemma with delayed terms]\label{lem2}
Suppose that assumptions B1, B2, and B4 hold. Let $x_j\in \rint \dom w$ for all $j\le k$ and $x_{k+1}$ be obtained via \eqref{frame1}. Let $L:=\max\limits_{k\geq 0}\sum\limits_{j\in J_k}L_j$ and
$$\Delta_k := L\cdot \ell(\tau+1) \sum_{j=k-\tau}^k  D_w(x_{j+1},x_j),$$
where $\tau$ is the smallest upper bound for all delayed
indexes. Then we have
\begin{equation*}
\Phi(x_{k+1})\leq \Phi(x)+\frac{1}{\alpha}D_w(x,x_k)-\frac{1}{\alpha}D_w(x,x_{k+1})-\frac{1}{\alpha}D_w(x_{k+1},x_k) + \Delta_k,\quad \forall x\in \dom w.
\end{equation*}
\end{lemma}

Based on Lemma \ref{lem2}, we derive the following sublinear convergence of the PLIAG method.

\begin{theorem}[Sublinear convergence without a growth condition]\label{mainresult0}
Let $L:=\max\limits_{k\geq 0}\sum\limits_{j\in J_k}L_j$,  where all $L_j$ are constants in assumption B1. Suppose that assumptions B1, B2, and B4 hold, and the step size $\alpha_k\equiv\alpha$ satisfies
\begin{equation}\label{step1}
0<\alpha\leq  \frac{2}{L\cdot \ell(\tau+1)\cdot(\tau+1)\cdot(\tau+2)},
\end{equation}
where $\ell(\cdot)$ is the function in assumption B4.
Define a function sequence of Lyapunov's type as
\begin{equation}
T_k(x):=\Phi(x_k)-\Phi(x)+L\cdot \ell(\tau+1)\cdot\sum_{i=1}^\tau i\cdot D_w(x_{k-\tau+i},x_{k-\tau+i-1}),\quad \forall k\geq 0,\ x\in \dom w.
\end{equation}
Then the PLIAG method converges sublinearly in the sense that
$$T_k(x)\leq \frac{1}{\alpha\cdot k}D_w(x,x_0), ~~\forall k\geq 1,\ x\in \dom w.$$
In particular, the PLIAG method attains a sublinear convergence in function values:
\begin{equation}\label{result020}
\Phi(x_k)-\Phi^* \leq \frac{1}{\alpha\cdot k}D_w(x^*,x_0),\quad \forall k\geq1,\ x^*\in\cX.
\end{equation}
\end{theorem}

When there is no delay in the PLIAG method (i.e., $\tau=0$), it follows from Theorem \ref{mainresult0} that for any $\alpha$ satisfying
$$0<\alpha \leq \frac{2}{L\cdot \ell(\tau+1)\cdot(\tau+1)\cdot(\tau+2)}= \frac{1}{L},$$
the result \eqref{result020} holds. In particular, when we set  $\alpha =1/L$, the following result holds
$$\Phi(x_k)-\Phi^* \leq \frac{L}{k}D_w(x^*,x_0),\ \forall k\geq1.$$
This  exactly recovers the sublinear convergence \cite[Theorem 4.1]{Teboull2018A} for first order methods.  To the best of our knowledge, Theorem \ref{mainresult0} is the first sublinear convergence result for incremental aggregated type first order methods with $\tau\geq 1$ including the PIAG method as a special case.

We next investigate the linear convergence of the PLIAG method under an additional assumption B3 based on Lemma \ref{lem2} and the following result given in \cite{Analysis2016Aytekin}.

\begin{lemma}\cite[Lemma 1]{Analysis2016Aytekin} \label{lem1}
Assume that the nonnegative sequences $\{V_k\}$ and $\{w_k\}$ satisfy
$$V_{k+1}\leq a V_k -b w_k+ c\sum_{j=k-k_0}^k w_j,\quad \forall k\ge0,$$
where $a\in(0,1)$, $b\geq 0$, $c\geq 0$, and $k_0\ge0$. Assume also that $w_k=0$ for all $k<0$, and the following condition holds:
\begin{equation}\label{condition}
\frac{c}{1-a}\frac{1-a^{k_0+1}}{a^{k_0}}\leq b.
\end{equation}
Then $V_k\leq a^k V_0$ for all $k\geq 0$.
\end{lemma}

\begin{theorem}[Linear convergence using a growth condition]\label{mainresult}
Let  $L:=\max\limits_{k\geq 0}\sum\limits_{j\in J_k}L_j$,  where all $L_j$ are constants in assumption B1. Suppose that assumptions B1-B4 hold, and the step size $\alpha_k\equiv\alpha$ satisfies
\begin{equation}\label{alpha0}
0<\alpha\leq \alpha_0:= \frac{\left(1+\frac{\mu}{L}\frac{1}{\ell(\tau+1)}\right)^{\frac{1}{\tau+1}}-1}{\mu},
\end{equation}
where $\mu$ is the constant in assumption B3 and $\ell(\cdot)$ is the function in assumption B4. Define a Lyapunov function as
\begin{equation}\label{Lya}
\Gamma_\alpha(x):=\Phi(x)-\Phi^*+\frac{1}{\alpha}\inf_{z\in\cX}D_w(z,x).\end{equation}
Then the PLIAG method converges linearly in the sense that
\begin{equation}\label{result01}
\Gamma_\alpha(x_k)\leq \left(\frac{1}{1+\alpha\mu}\right)^k \Gamma_\alpha(x_0),\quad \forall k\geq0.
\end{equation}
In particular, the PLIAG method attains a globally linear convergence in terms of function values:
\begin{equation}\label{result02}
\Phi(x_k)-\Phi^* \leq \left(\frac{1}{1+\alpha\mu}\right)^k \Gamma_\alpha(x_0),\quad \forall k\geq0,
\end{equation}
and a globally linear convergence in terms of Bregman distances to the optimal solution set:
\begin{equation}\label{result03}
\inf_{z\in\cX}D_w(z,x_k)\leq \alpha\Gamma_\alpha(x_0) \left(\frac{1}{1+\alpha\mu}\right)^{k+1}, \quad \forall k\geq0.
\end{equation}
Furthermore, if $\alpha=\alpha_0$ where $\alpha_0$ is given in \eqref{alpha0}, then
\begin{equation}\label{result04}
\Gamma_{\alpha_0}(x_k)\leq \left(1-\frac{1}{[\ell(\tau+1)Q+1](\tau+1)}\right)^k \Gamma_{\alpha_0}(x_0),\quad \forall k\geq0,
\end{equation}
where $Q:=L/\mu$ is used to denote the condition number of the problem.
\end{theorem}

A key ingredient in Theorem \ref{mainresult} is the Lyapunov function \eqref{Lya}, which includes two terms: the function value difference $\Phi(x)-\Phi^*$ and the Bregman distance to the optimal solution set $\inf_{z\in \cX}D_w(z,x)$.
It is different from the Lyapunov function of using only one of the two terms, constructed in the existing convergence analysis for the PIAG method.
Based on the new Lyapunov function, we give a unified analysis for globally linear convergence in terms of both the function values and Bregman distances to the optimal solution set beyond the standard assumptions. If we let $\cI_k\equiv\emptyset$, $\cQ=\RR^d$, and $w(x)=\frac{1}{2}\|x\|^2$, then the PLIAG method reduces to the PIAG method whose linear convergence of the objective function values and the distance of the iterates to the optimal solution set are studied respectively in  \cite{Global2016Vanli,A2016Vanli} and \cite{Analysis2016Aytekin} under the standard assumptions A1-A3. Our results can recover these two classes of convergence under assumptions A1, A2, and A3a (weaker than the strong convexity assumption of $F$ in the literature).
In particular,
\begin{itemize}
  \item The result \eqref{result02} can recover the linear convergence in \cite{Global2016Vanli} and moreover, if the upper bound for all delayed indexes is chosen as $\tau \leq 47$ and $L\ge \mu$, then by choosing $\ell(\cdot)$ as the identity function, the result \eqref{result04} recovers a better linear convergence rate given in \cite{A2016Vanli}:
  \begin{equation}\label{best}
\Phi(x_k)-\Phi^*\leq \left(1-\frac{1}{49 Q(\tau+1)}\right)^k \Gamma_{\alpha_0}(x_0),\quad \forall k\ge0.
\end{equation}
Here we need to point out that when $\tau > 47$, the result \eqref{best} is better than the result \eqref{result04}. This leaves us a question whether one can derive the stronger result \eqref{best} for any nonnegative $\tau$ under the modified assumptions B1-B4. We will consider it for future research.
\item In the setting of the problem considered in \cite{Analysis2016Aytekin}, the optimal solution set $\cX=\{x^*\}$ and then the result \eqref{result03} reads as
  \begin{equation*}
\|x_k-x^*\|^2 \leq 2\alpha \Gamma_{\alpha}(x_0) \left(\frac{1}{1+\alpha\mu}\right)^{k+1},\quad \forall k\ge0,
\end{equation*}
which is actually the linear convergence result in \cite{Analysis2016Aytekin} up to a constant.
\end{itemize}

When letting $I_k\equiv\emptyset$, $N=1$, and $\tau_k^n\equiv 0$,  the PLIAG method reduces to the NoLips algorithm recently studied in \cite{Bauschke2016A} where only the globally sublinear convergence was established (a specialization of Theorem \ref{mainresult0}). As a supplement, Corollary \ref{cor41} gives the linear convergence rate of the NoLips algorithm which follows from Theorem \ref{mainresult} immediately. In an independent work \cite[Proposition 4.1]{Teboull2018A}, a linear convergence rate of the NoLips algorithm is established under a relative convexity condition (i.e., $F-\sigma w$ is convex on $\rint \dom w$ for some $\sigma>0$)
which guarantees assumption B3 when the Bregman distance $D_w(x,y)$ is symmetric (see discussions in Appendix \ref{secB}) but not verse visa.

\begin{corollary}\label{cor41}
Consider the NoLips algorithm proposed in \cite{Bauschke2016A}. If assumptions B1--B3 hold, then it attains a globally linear convergence in the sense of \eqref{result01}.
\end{corollary}


When $\cI_k=\{i_k\}$, $h(x)= \delta_\cQ(x)$, and $w(x)=\frac{1}{2}\|x\|^2$,  assumptions B1 and B3 reduce to assumptions A1 and A3a respectively, and the PLIAG iterate \eqref{frame1} becomes
\begin{equation}\label{IAP-C}
x_{k+1}=\arg\min_{x\in \cQ}\Big\{f_{i_k}(x)+ \Big\langle \sum_{j\neq i_k}\nabla f_j(x_{k-\tau_k^j}),x \Big\rangle +\frac{1}{2\alpha_k}\|x-x^k\|^2 \Big\}.
\end{equation}
We call \eqref{IAP-C} the constrained IAP method. The IAP method (i.e., the constrained IAP method with $\cQ=\RR^d$) has been studied in \cite{Bertsekas2015incre} in which the linear convergence rate of the IAP method was established under standard assumptions A1 and A3. In \cite{Bertsekas2015incre}, Bertsekas also pointed out that the proof of the IAP method breaks down when $\cQ$ is not the whole space since a critical inequality fails, and leaved it as an open question. The following corollary, which follows from Theorems \ref{mainresult0} and \ref{mainresult}, gives the convergence results of the constrained IAP method under assumption A1 and assumption A3a (weaker than assumption A3).
\begin{corollary}\label{cor4.2}
If assumption A1 holds, then the constrained IAP method attains a globally sublinear convergence. If, in addition, assumption A3a holds, then it attains a globally linear convergence.
\end{corollary}

%

\section{Linear convergence under H\"{o}lderian growth condition}\label{se5}

The convergence results developed in Section \ref{se4} also apply to the PIAG method since
it is a special case of the PLIAG method. More interestingly, we find that when directly analyzing the convergence of the PIAG method, as discussed in the last paragraph of Subsection \ref{subse:growth}, a locally improved linear convergence can be derived under the H\"{o}lderian growth condition \eqref{pgrow}. This section can be seen as a further research for the convergence of the PIAG method, which is of independent interest.

The following lemma, which can be viewed as a generalization of Lemma \ref{lem1}, is key to get our main result.

\begin{lemma}\label{lem7}
Assume that the nonnegative sequences $\{V_k\}$ and $\{w_k\}$ satisfy
\begin{equation}\label{shock2}
d V^\theta_{k+1} +a  V_{k+1}\leq a  V_k  -b w_k+ c\sum_{j=k-k_0}^k w_j, \quad \forall k\ge0,
\end{equation}
where $a,\theta\in(0,1]$, $b$, $c\geq 0$, $d\geq 0$, $a+d=1$, $V_0\leq 1$, and $k_0\ge0$. Assume also that $w_k=0$ for all $k<0$ and condition \eqref{condition} holds.
Then $V_k^\rho \leq a^k V_0$ for all $k\geq 0$, where $\rho:=(1-a)\theta +a$.
\end{lemma}

We are ready to give the main result of this section.
\begin{theorem}\label{Thm6}
Suppose that assumptions A0, A1, A2, and A3b hold. Assume that $d(x_0,\cX)\leq 1$. If the step size $\alpha_k\equiv\alpha$ satisfies
\begin{equation}\label{step1}
0<\alpha\leq \alpha_0:= \frac{\left(1+\frac{\mu }{(\tau+1)L} \right)^{\frac{1}{\tau+1}}-1}{\mu},
\end{equation}
then the PIAG method converges linearly in the sense that
\begin{equation}\label{resultad2}
d^2(x_k,\cX)\leq   \left(\frac{1}{1+\alpha\mu}\right)^\frac{k}{\eta} d^2(x_0,\cX),\quad \forall k\geq0,
\end{equation}
where $\eta :=\frac{\alpha \mu  }{1+\alpha\mu  }\theta +\frac{1}{1+\alpha\mu  }$ and $\theta\in (0,1]$ is given in assumption A3b. Moreover, if $\alpha=\alpha_0$, then
\begin{equation}\label{asym2}
d^2(x_k,\cX)\leq  \left( 1-\frac{1}{(\tau+1)\rho_0+(\tau+1)^2\eta_0Q } \right)^k d^2(x_0,\cX), \quad \forall k\geq0,
\end{equation}
where $Q:=L/\mu$ and $\eta_0 :=\frac{\alpha_0 \mu }{1+\alpha_0\mu  }\theta +\frac{1}{1+\alpha_0\mu }.$
\end{theorem}

Combining Theorems \ref{mainresult} with \ref{Thm6} yields the global convergence rate of the PIAG method.

\begin{corollary}
Suppose that assumptions A0, A1, A2, A3a, and A3b hold. If the step size $\alpha_k\equiv\alpha$ satisfies $0<\alpha\leq \alpha_0$, then the linear convergence rate of the PIAG method is $\frac{1}{1+\alpha\mu}$. Moreover, when an iteration point $x_{k_0}$ is such that $d(x_{k_0},\cX)\le 1$, the convergence rate of the truncated sequence $\{x_k\}_{k\ge k_0}$ is improved to $\left(\frac{1}{1+\alpha\mu}\right)^{\frac{1}{\eta}}$, where $\alpha_0>0$ and $\eta\in (0,1]$ are defined in Theorem \ref{Thm6}.
\end{corollary}




\section{Concluding remarks}\label{se6}

In this paper, we proposed a unified algorithmic framework for minimizing the sum of a convex function consisting of additive relatively smooth convex component functions and a possibly non-smooth convex regularization function, over an abstract feasible set. Our proposed algorithm includes the IAG, IAP, PG, and PIAG methods as special cases, and it also includes some novel interesting schemes such as the non-Euclidean IAP and PIAG methods. To the best of our knowledge, we established the first sublinear convergence result for incremental aggregated type first order methods with positive delayed parameters.  Moreover, by introducing the Bregman distance growth condition and employing the Lipschitz-like/convexity condition, we obtained a group of linear convergence results for the unified algorithm. Many convergence results are new even when specialized to the existing algorithms. The key idea behind our proof is to construct a certain Lyapunov function by embedding the Bregman distance growth condition into a descent-type lemma.

Although in Appendix B we gave a nontrivial example such that assumptions B1-B3 hold but assumptions A1 and A3 fail, it is more interesting to find some practical examples satisfying the modified assumptions B1-B4 simultaneously but do not falling into the more restrictive assumptions A1-A3. More generally, it is interesting to study how much the modified assumptions, which are related to each other with the same function $w$, are weaker than the standard assumptions. Moreover, it is also interesting to relax the convexity requirement in the PLIAG framework to develop corresponding linear convergence theory for nonconvex optimization as done in the recent publications \cite{Bauschke2019On,Peng2018non} for special cases of the PLIAG method for solving nonconvex optimization. These are beyond the scope of this paper and we leave the related work for future research.

Finally, we remark that the proof technique developed in this paper will find more applications in other types of incremental aggregated methods, including randomized and accelerated versions of the PIAG method. Indeed,  our proof technique has  been modified to analyze linear convergence of nonconvex and inertial PIAG methods \cite{Peng2018non,zhang2017prox}. However, it should be noted that it is not trivial to extend our proof technique to analyze the inertial version of the PLIAG method:
\begin{align*}
& z_{k+1}=\arg\min_{x\in \RR^d}\Big\{h(x)+ \sum_{i\in \cI_k} f_{i}(x)+ \Big\langle \sum_{j\in \cJ_k}\nabla f_j(x_{k-\tau_k^j}),x \Big\rangle +\frac{1}{\alpha_k}D_w(x,x_k)\Big\}, \\
& x_{k+1}=z_{k+1}+\eta_k \cdot (z_{k+1}-z_k),
\end{align*}
where $\eta_k$ is a positive inertial parameter. The potential difficulty is that some inequalities for the Euclidean distance fail to hold for the Bregman distance. The convergence analysis of the inertial PLIAG method is beyond the scope of this paper and we leave it for future research as well.

\section*{Acknowledgements}

The authors thank the area editor, the associate editor, and the two anonymous
referees for their helpful comments and constructive suggestions. The authors are also grateful to Prof. Yair Censor for bringing his early work to their attention.

%
%

 \setcounter{section}{0}
\renewcommand\thesection{\Alph{section}}

\begin{appendices}

\section{Examples satisfying quadratic growth condition}

We introduce two examples whose objective function fails to be strongly convex but satisfies the quadratic growth condition. Then our derived results in Section \ref{se4} show that the PIAG method is a suitable algorithm for solving them; see the discussions after Theorem \ref{mainresult}. More examples can be found in \cite{Drusvyatskiy2016Error}.
\begin{example}\cite[Lemma 10]{Bolte2015From}
Consider the least squares optimization with $\ell_1$-norm regularization:
 $$\min_{x\in \RR^d}\quad \Phi(x):= \frac{1}{2}\|Ax-b\|^2+\lambda \|x\|_1,$$
where $\lambda>0$ is a regularization parameter, $A$ and $b$ are appropriate matrix and vector respectively. Let $R>\frac{\|b\|^2}{2\lambda}$. Then, there exists a constant $\mu>0$ such that
 \begin{equation}\label{ex}
 \Phi(x)-\Phi^*\geq \mu d^2(x,\cX),~~\forall x\in \Omega=\{x\in\RR^d: \|x\|_1\leq R\}.
 \end{equation}
 Let $w(x)=\frac{1}{2}\|x\|^2,\ \cQ=\RR^d$,  $F(x)=\frac{1}{2}\|Ax-b\|^2$, and $h(x)=\|x\|_1+\delta_\Omega(x)$. In this setting, the quadratic growth condition is actually \eqref{ex}.
\end{example}

\begin{example}
The following is a dual model appeared in compressed sensing
$$\min_{x\in \RR^d}\quad F(x):= -b^Tx+\frac{\alpha}{2}\| \textrm{shrink}_\mu(A^Tx)\|^2.$$
Here the parameter $\alpha$, matrix $A$, and vector $b$ are given, and the operator ``shrink" is defined by
 $$\textrm{shrink}_\mu(s)=\textrm{sign}(s)\max\{|s|-\mu,0\},$$
where  $\textrm{sign}(\cdot), |\cdot|$, and $\max\{\cdot, \cdot\}$ are component-wise operations for vectors. Note that the smooth part $F$ in the problem above can be written into the sum of many component functions.

The objective function $F$ was shown in \cite{Lai2013Aug} to be restricted strongly convex on $\RR^d$, and hence satisfies the quadratic growth condition with $\mathcal{Q}=\RR^d$ due to their equivalence \cite{Zhang2015The,zhang2016new}.
\end{example}

\section{Discussions on Bregman distance growth condition}\label{secB}

We investigate sufficient conditions for Bregman distance growth condition to hold and give a class of functions satisfying this growth condition.

\begin{theorem}\label{thm-suffi}
The Bregman distance growth condition holds under any one of the two conditions:
\begin{itemize}

  \item[(C1).] $\Phi$ satisfies the quadratic growth condition and $\nabla w$ is Lipschitz continuous  on $\dom w$.

  \item[(C2).] $h$ is convex,  $\dom h\cap \rint\dom w \not= \emptyset$, and $F$ satisfies
\begin{equation}\label{scnew}
F(y)\geq F(x)+\langle \nabla F(x),y-x\rangle +\mu \cdot D_w(x,y),~~\forall x, y\in \rint\dom{w}.
\end{equation}
\end{itemize}
\end{theorem}
\begin{proof}

Assume that (C1) holds. Let $\bar{y}$ stand for the projection of $y$ onto $\cX$. It then follows from the quadratic growth condition of $\Phi$ that there exists $\mu>0$ such that
$$\Phi(y)-\Phi^* \geq   \frac{\mu}{2}d^2(y,\cX) =\frac{\mu}{2}\|y-\bar{y}\|^2,\quad \forall y\in \dom w.
$$
Since $\nabla w$ is Lipschitz continuous, it is easy to verify that there exists $L>0$ such that
$$\frac{\mu}{2}\|y-\bar{y}\|^2\geq \frac{\mu}{L}  D_w(\bar{y}, y)\le \frac{\mu}{L} \inf_{z\in \cX} D_w(z, y),\quad \forall y\in \rint\dom w.$$
Thus, the desired result follows by combining the last two inequalities immediately.

Assume that (C2) holds. The convexity of $h$ together with \eqref{scnew} implies that for any $v\in \partial h (x)$,
\begin{equation}\label{app-ine}
F(y) +h(y)\geq F(x)+h(x)+\langle \nabla F(x)+v,y-x\rangle +\mu \cdot D_w(x,y), ~~\forall x, y\in \rint\dom{w}.
\end{equation}
Pick up an optimal solution $\bar{x}\in\cX$. Using Fermat's rule and the assumption that $ \dom h\cap \rint\dom w \not= \emptyset$, it follows that there exists $\bar{v}\in \partial h(\bar{x})$ such that
\[
\langle \nabla F(\bar{x})+\bar{v}, y-\bar{x}\rangle \ge 0,\quad \forall y\in \dom w.
\]
Recalling that $F(\bar{x})+h(\bar{x})=\Phi^*$, substituting $x=\bar{x}$ and $v=\bar{v}$ in \eqref{app-ine} indicates
$$
\Phi(y)-\Phi^*\geq \mu \cdot D_w(\bar{x},y) \geq \mu\cdot  \inf_{z\in \cX}D_w(z,y),~~\forall y\in \rint\dom{w},
$$
which is exactly the Bregman distance growth condition.
\end{proof}

It should be pointed out that condition \eqref{scnew} is different from the $\mu$-strong convexity relative to $w$ that was introduced in \cite{Lu2016relatively}:
$$F(y)\geq F(x)+\langle \nabla F(x),y-x\rangle +\mu \cdot D_w(y,x), ~~~\forall x, y\in \rint\dom{w}.$$
These two conditions are different due to the possible nonsymmetry of the Bregman distance. To link them, we need a measure for the lack of symmetry in $D_w$ introduced in \cite{Bauschke2016A}.
\begin{definition}\label{def1}
Given a Legendre function $w:\RR^d\rightarrow (-\infty,\infty]$, its symmetry coefficient is defined by
$$\alpha(w):=\inf \left\{\frac{D_w(x,y)}{D_w(y,x)}: x, y\in \rint\dom{w}, x\neq y\right\}\in [0,1].$$
\end{definition}
If $a(w)\neq 0$, then by the definition of the symmetry coefficient, it is not hard to see that
$$\alpha(w)D_w(y,x)\leq D_w(x,y) \leq \alpha(w)^{-1}D_w(y,x), ~~\forall x, y\in \rint\dom{w},$$
which implies that condition \eqref{scnew} and the relatively strong convexity are equivalent up to a constant.

We now revisit a class of convex functions studied in \cite{Lu2016relatively}:
$$
F(x)= \frac{1}{4}\|Ex\|_2^4+\frac{1}{4}\|Ax-b\|_4^4+\frac{1}{2}\|Cx-d\|_2^2,
$$
where $E,A,C$ and $b,d$ are appropriate matrices and vectors respectively, $\|\cdot\|_p$ stands for the $\ell_p$ norm for $p>1$. If there is no confusion, we  use $\|\cdot\|$ to replace $\|\cdot\|_2$ for simplicity. It is not hard to see that $F$ is not strongly convex and $\nabla F$ is not Lipschitz continuous. Actually, by assuming that the smallest singular values of $E$ and $C$, denoted by $\sigma_E$ and $\sigma_C$ respectively, are positive, it has been shown that $F(x)$ is $L$-smooth (and hence satisfies the Lipschitz-like/convexity condition) and $\mu$-strongly convex relative to
\begin{equation}\label{w}
w(x)=\frac{1}{4}\|x\|^4+\frac{1}{2}\|x\|^2
\end{equation}
with $L:=3\|E\|^4+3\|A\|^4+6\|A\|^3\|b\|_2+3\|A\|^2\|b\|^2+\|C\|^2$ and $\mu:=\min\{\frac{\sigma_E^4}{3},\sigma_C^2\}$, where this function in \eqref{w} is also used in \cite{bolte2017first} to show the Lipschitz-like/convexity of a nonconvex function. In Proposition \ref{coposi} below, we show that the symmetry coefficient of $w$ is positive, where $w$ is defined in \eqref{w}. Thus, based on Theorem \ref{thm-suffi}, a class of functions satisfying the Bregman distance growth condition can be constructed in the following form:
$$\Phi(x)=\frac{1}{4}\|Ex\|_2^4+\frac{1}{4}\|Ax-b\|_4^4+\frac{1}{2}\|Cx-d\|_2^2 + h(x).$$
Thus, when applying the PLIAG method with $\tau=0$ to solve the minimization problem $``\min_{x\in \RR^d} \Phi(x)"$, the linear convergence is available from Theorem \ref{mainresult}.

\begin{proposition}\label{coposi}
Let $w(x)=\frac{1}{4}\|x\|^4 + \frac{1}{2}\|x\|^2$. Then $\alpha(w)\geq \frac{1}{5}$. More generally, let $\tilde{w}(x)=\frac{\beta}{4}\|x\|^4 + \frac{\gamma}{2}\|x\|^2$ with $\beta, \gamma >0$, then $\alpha(\tilde{w})\geq \frac{1}{5}\frac{\min(\beta, \gamma)}{\max(\beta, \gamma)}$.
\end{proposition}

\begin{proof}
Let $w_1(x):=\frac{1}{4}\|x\|^4$ and $w_2(x):=\frac{1}{2}\|x\|^2$.  Then $w(x)=w_1(x)+w_2(x)$.
By the linear additivity of the Bregman distance, it follows that
$$D_w(x,y)=D_{w_1+w_2}(x,y)=D_{w_1}(x,y)+D_{w_2}(x,y).$$
After some simple calculus, we get $D_{w_1}(x,y)=\frac{1}{4}\|x\|^4+\frac{3}{4}\|y\|^4-\|y\|^2\langle x,y\rangle$ and $D_{w_2}(x,y)=\frac{1}{2}\|x-y\|^2$. It follows from the definition of the symmetry coefficient that
\begin{align*}
\alpha(w)=&\inf_{x\neq y}\frac{\frac{1}{4}\|x\|^4+\frac{3}{4}\|y\|^4-\|y\|^2\langle x,y\rangle+\frac{1}{2}\|x-y\|^2}{\frac{1}{4}\|y\|^4+\frac{3}{4}\|x\|^4-\|x\|^2\langle x,y\rangle+\frac{1}{2}\|x-y\|^2}\\
=&\inf_{x\neq y}\frac{\frac{1}{4}(\|x\|^2-\|y\|^2)^2+ \frac{1}{2}\|x-y\|^2(1+\|y\|^2)}{\frac{1}{4}(\|x\|^2-\|y\|^2)^2+ \frac{1}{2}\|x-y\|^2(1+\|x\|^2)}\\
=&\inf_{\|y\|<\|x\|}\frac{\frac{1}{4}(\|x\|^2-\|y\|^2)^2+ \frac{1}{2}\|x-y\|^2(1+\|x\|^2) +\frac{1}{2}\|x-y\|^2(\|y\|^2-\|x\|^2)}{\frac{1}{4}(\|x\|^2-\|y\|^2)^2+ \frac{1}{2}\|x-y\|^2(1+\|x\|^2)}\\
=& 1-\sup_{\|y\|<\|x\|} \frac{\frac{1}{2}\|x-y\|^2(\|x\|^2-\|y\|^2)}{\frac{1}{4}(\|x\|^2-\|y\|^2)^2+ \frac{1}{2}\|x-y\|^2(1+\|x\|^2)}\\
=& 1-\sup_{\|y\|<\|x\|}\left(\frac{1}{2}\frac{\|x\|^2-\|y\|^2}{\|x-y\|^2}+\frac{1+\|x\|^2}{\|x\|^2-\|y\|^2}\right)^{-1},
\end{align*}
where the third equality follows from the observation that the minimization problem attains  an optimal solution when $\|y\|< \|x\|$. Denote
\[
R(x,y):= \frac{1}{2}\frac{\|x\|^2-\|y\|^2}{\|x-y\|^2}+\frac{1+\|x\|^2}{\|x\|^2-\|y\|^2}.
 \]
It suffices to show that $R(x,y)\geq \frac{5}{4}$ under the constraint $\|y\|<\|x\|$. We can find $\eta>0$ such that $\|x\|^2=\|y\|^2+\eta^2$. Then, $\|x\|^2\leq (\|y\|+\eta)^2$ and hence $\|x\|\leq \|y\|+\eta$, which implies $\|x-y\|\leq \|x\|+\|y\|\leq 2\|y\|+\eta.$ Therefore, by letting $\mu:= \frac{\|y\|}{\eta}$ we have
\begin{align*}
R(x,y)\geq &\frac{1}{2}\frac{\eta^2}{(2\|y\|+\eta)^2}+\frac{1+\|y\|^2+\eta^2}{\eta^2}\\
\geq &\inf_{\mu\ge0,\eta\ge0}\left( \frac{1}{2}\frac{1}{(2\mu+1)^2}+1+\mu^2+\frac{1}{\eta^2}\right)\\
\geq &\inf_{\mu\ge0,\eta\ge0}\left(  \frac{1}{4}\big(\frac{1}{4\mu^2+1}+ 4\mu^2+1\big)+\frac{3}{4}+\frac{1}{\eta^2}\right)\\
\geq & \inf_{\mu\ge0,\eta\ge0}\left( \frac{1}{2}+\frac{3}{4}+\frac{1}{\eta^2}\right)\geq \frac{5}{4}.
\end{align*}
  For the general case, we derive that
\begin{align*}
\alpha(\tilde{w})=&\inf_{x\neq y} \frac{\beta D_{w_1}(x,y)+\gamma D_{w_2}(x,y)}{\beta D_{w_1}(y,x)+\gamma D_{w_2}(y,x)}\\
\geq & \frac{\min(\beta, \gamma)}{\max(\beta, \gamma)} \inf_{x\neq y} \frac{D_{w_1}(x,y)+  D_{w_2}(x,y)}{ D_{w_1}(y,x)+ D_{w_2}(y,x)}\\
= & \frac{\min(\beta, \gamma)}{\max(\beta, \gamma)} \inf_{x\neq y} \frac{D_{w}(x,y)}{D_{w}(y,x)}\\
= & \frac{\min(\beta, \gamma)}{\max(\beta, \gamma)}\alpha(w).
\end{align*}
The proof is complete by noting that $\alpha(w)\ge \frac{1}{5}$.
\end{proof}

We remark that the function $\tilde{w}(x)=\frac{\beta}{4}\|x\|^4 + \frac{\gamma}{2}\|x\|^2$ with $\beta, \gamma >0$  was used for non-Euclidean first order methods to
solve a class of low-rank minimization problems; see, e.g., \cite{Dragomir2019Quartic} where the function $\tilde{w}(x)$ is called quartic universal kernel.
It is easy to verify that in the matrix case, Proposition \ref{coposi} is also true. This might help us design more aggressive step size strategies for low rank minimization problems discussed in \cite{Dragomir2019Quartic,Li2019Provable}.

\section{Proofs of the propositions}

\noindent{\bf Proof of Proposition \ref{pro1}.}
We only give the proof for the case when $\tau>0$ since the proof for the case when $\tau=0$ follows similarly.

First we observe that the objective function $\Phi_k$ in \eqref{frame1} is strictly convex. Thus this problem has at most one minimizer. Let $x_j\in \rint\dom{w}$ for any $j\le k$ and $x\in \rint\dom{w}$. Then by assumptions B1 and B4, it follows that
\begin{align}\label{C.1}
f_j(x)\leq &f_j(x_{k-\tau_k^j})+\langle \nabla f_j(x_{k-\tau_k^j}), x-x_{k-\tau_k^j}\rangle + L_j\cdot D_w(x,x_{k-\tau_k^j}) \nonumber\\
\leq  &f_j(x_{k-\tau_k^j})+\langle \nabla f_j(x_{k-\tau_k^j}), x-x_{k-\tau_k^j}\rangle +\ell(2)L_j\cdot D_w(x,x_k)+\ell(2)L_j\cdot D_w(x_k,x_{k-\tau_k^j}).
\end{align}
Summing \eqref{C.1} over all $j\in J_k$ and noting that $\alpha_k\leq \frac{1}{\ell(2)\sum_{j\in J_k}L_j}$, we have
\begin{align*}
\sum_{j\in \cJ_k}f_j(x)\leq & C_k+ \langle \sum_{j\in \cJ_k} \nabla f_j(x_{k-\tau_k^j}), x\rangle  + \ell(2)\sum_{j\in \cJ_k} L_j\cdot D_w(x,x_k)\nonumber\\
\leq &C_k+ \langle \sum_{j\in \cJ_k} \nabla f_j(x_{k-\tau_k^j}), x\rangle  + \frac{1}{\alpha_k}\cdot D_w(x,x_k),
\end{align*}
where $$C_k:=\sum_{j\in \cJ_k} f_j(x_{k-\tau_k^j})- \langle \sum_{j\in \cJ_k} \nabla f_j(x_{k-\tau_k^j}), x_{k-\tau_k^j}\rangle + \ell(2)\sum_{j\in \cJ_k} L_j\cdot D_w(x_k,x_{k-\tau_k^j}).$$
This above inequality leads to the following relationship between $\Phi$ and $\Phi_k$:
\begin{equation}\label{big}
\Phi(x)\leq \Phi_k(x)+C_k,\quad \forall x\in \rint\dom{w}.
\end{equation}
Due to the nonemptiness and compactness of the optimal solution set $\cX$, we have that  $\Phi+\delta_\mathcal{Q}$ is level-bounded; see, e.g., \cite[Corollary 8.7.1]{Rock1970convex}.  Then by \eqref{big}, we have that $\Phi_k+\delta_\cQ$ is also level-bounded. Thus the optimal solution set of minimizing $\Phi_k$ over $\cQ$ is nonempty by Weierstrass's theorem. Therefore problem \eqref{frame1} has only one solution, say $x_{k+1}$. The fact that $x_{k+1}\in \rint\dom{w}$  can be seen from the optimality of $x_{k+1}$ and the fact that $\partial w(z)=\emptyset$ for any $z\notin \rint\dom{w}$ (see \cite[Theorem 26.1]{Rock1970convex}).

\bigskip

\noindent{\bf Proof of Lemma \ref{lem2}.}
Since $f_j(x)$ is convex and $L_j$-smooth relative to $w$, it follows that
\begin{align}
f_j(x_{k+1})\leq &f_j(x_{k-\tau_k^j})+\langle \nabla f_j(x_{k-\tau_k^j}), x_{k+1}-x_{k-\tau_k^j}\rangle + L_j\cdot D_w(x_{k+1},x_{k-\tau_k^j}) \nonumber\\
\leq  &f_j(x)+\langle \nabla f_j(x_{k-\tau_k^j}), x_{k+1}-x\rangle + L_j\cdot D_w(x_{k+1},x_{k-\tau_k^j}). \label{Lp}
\end{align}
For simplicity, we denote
$$s_k:=\sum_{j\in\cJ_k}\nabla f_j(x_{k-\tau_k^j}).$$
Summing \eqref{Lp} over all $j\in \cJ_k$ and using the definition of $s_k$ yield
\begin{equation}\label{m1}
\sum_{j\in\cJ_k}f_j(x_{k+1})\leq \sum_{j\in\cJ_k}f_j(x)+\langle s_k, x_{k+1}-x\rangle +\sum_{j\in\cJ_k} L_j\cdot D_w(x_{k+1},x_{k-\tau_k^j}).
\end{equation}
By the optimality of $x_{k+1}$ and Fermat's rule, we have
\begin{equation}\label{m01}
-s_k-\frac{1}{\alpha}(\nabla w(x_{k+1})-\nabla w(x_k)) \in \partial h(x_{k+1})+ \sum_{i\in\cI_k}\nabla f_i(x_{k+1}).
\end{equation}
Using the subgradient inequality for the convex function $h(x)+\sum_{i\in\cI_k} f_i(x)$ at $x_{k+1}$ implies
\begin{align}
& h(x_{k+1})+\sum_{i\in\cI_k}\nabla f_i(x_{k+1})\nonumber \\
\leq & h(x)+\sum_{i\in\cI_k}\nabla f_i(x)+ \langle s_k+\frac{1}{\alpha}(\nabla w(x_{k+1})-\nabla w(x_k)),x- x_{k+1} \rangle \nonumber\\
= & h(x)+\sum_{i\in\cI_k}\nabla f_i(x) +\langle s_k ,x- x_{k+1}\rangle +\frac{1}{\alpha}\langle \nabla w(x_{k+1})-\nabla w(x_k),x- x_{k+1} \rangle \nonumber\\
=& h(x)+\sum_{i\in\cI_k}\nabla f_i(x) +\langle s_k ,x- x_{k+1}\rangle +\frac{1}{\alpha}D_w(x,x_k)-\frac{1}{\alpha}D_w(x,x_{k+1})-\frac{1}{\alpha}D_w(x_{k+1},x_k)\label{p2},
\end{align}
where the last equality follows from the three-point identity of the Bregman distance.
Adding \eqref{p2} to \eqref{m1} implies
\begin{equation}\label{pr}
\Phi(x_{k+1})\leq \Phi(x)+\frac{1}{\alpha}D_w(x,x_k)-\frac{1}{\alpha}D_w(x,x_{k+1})-\frac{1}{\alpha}D_w(x_{k+1},x_k)+\sum_{j\in\cJ_k} L_j\cdot D_w(x_{k+1},x_{k-\tau_k^j}).
\end{equation}
Recalling that $\tau_k^j$ is bounded above by $\tau$ and using assumption B4, it follows that
$$D_w(x_{k+1},x_{k-\tau_k^j})\leq \ell(\tau_k^j+1)\sum_{j=k-\tau_k^j}^k D_w(x_{j+1},x_j)\leq \ell(\tau+1)\sum_{j=k-\tau}^kD_w(x_{j+1},x_j).$$
Summing the above inequality over $j\in\cJ_k$ implies
\begin{align*}
\sum_{j\in\cJ_k} L_j\cdot D_w(x_{k+1},x_{k-\tau_k^j})\leq & \ell(\tau+1) \sum_{j\in\cJ_k} L_j\sum_{j=k-\tau}^kD_w(x_{j+1},x_j)\nonumber\\
&\leq L \cdot\ell(\tau+1)\sum_{j=k-\tau}^kD_w(x_{j+1},x_j).
\end{align*}
This together with \eqref{pr} leads to the desired result.

\bigskip

\noindent{\bf Proof of Theorem \ref{mainresult0}.} First by Proposition \ref{pro1} and the choice of the step size \eqref{step1}, the generated sequence is well-defined. The rest of the proof will rely on the following formulations:
\begin{align}
\sum_{j=k-\tau}^k a_j = & a_{k-\tau}+\sum_{i=2}^{\tau+1} a_{k-\tau+i-1} \nonumber\\
= &   a_{k-\tau}+\sum_{i=1}^{\tau} a_{(k+1)-\tau+i-1}\nonumber\\
= & a_{k-\tau}+\sum_{i=1}^{\tau} (i+1) a_{k-\tau+(i+1)-1}- \sum_{i=1}^{\tau} i a_{(k+1)-\tau+i-1}\nonumber\\
= & a_{k-\tau}+\sum_{i=2}^{\tau+1} i a_{k-\tau+i-1}- \sum_{i=1}^{\tau} i a_{(k+1)-\tau+i-1}\nonumber\\
= & \sum_{i=1}^{\tau+1} i a_{k-\tau+i-1}- \sum_{i=1}^{\tau} i a_{(k+1)-\tau+i-1} \label{formula1}\\
= & (\tau+1)a_k+ \sum_{i=1}^{\tau} i a_{k-\tau+i-1}- \sum_{i=1}^{\tau} i a_{(k+1)-\tau+i-1}\label{formula2}.
\end{align}
Letting $a_j:=D_w(x_{j+1},x_j)$ and using formulation \eqref{formula2}, we have
\begin{eqnarray}
\sum_{j=k-\tau}^k D_w(x_{j+1},x_j)&=& (\tau+1)D_w(x_{k+1},x_k)+ \sum_{i=1}^{\tau} i D_w(x_{k-\tau+i},x_{k-\tau+i-1})\nonumber\\
&& - \sum_{i=1}^{\tau} i D_w(x_{(k+1)-\tau+i},x_{(k+1)-\tau+i-1}).\label{decom}
\end{eqnarray}
Invoking Lemma \ref{lem2} with $x=x_k$ implies
\begin{equation*}
\Phi(x_{k+1})\leq \Phi(x_k)-\frac{1}{\alpha}D_w(x_k,x_{k+1})-\frac{1}{\alpha}D_w(x_{k+1},x_k) + \Delta_k.
\end{equation*}
This together with \eqref{decom} and the definition of $T_k(x)$ implies
$$T_{k+1}(x)+\left(\frac{1}{\alpha}-L\cdot \ell(\tau+1)\cdot (\tau+1)\right)D_w(x_{k+1},x_k)\leq T_k(x)-\frac{1}{\alpha}D_w(x_k,x_{k+1})\leq T_k(x).$$
By the choice of $\alpha$, it is easy to see that $\frac{1}{\alpha}\geq L\cdot \ell(\tau+1)\cdot (\tau+1)$. Then the relation that $T_{k+1}(x)\leq T_k(x)$ follows immediately, i.e., the function sequence $\{T_k(x)\}_{k\geq 0}$ is nonincreasing.

On the other hand, letting $a_j:=D_w(x_{j+1},x_j)$ and using formulation \eqref{formula1}, we have
\begin{equation*}
\sum_{j=k-\tau}^k D_w(x_{j+1},x_j)=  \sum_{i=1}^{\tau+1} i D_w(x_{k-\tau+i},x_{k-\tau+i-1})- \sum_{i=1}^{\tau} i D_w(x_{(k+1)-\tau+i},x_{(k+1)-\tau+i-1}).
\end{equation*}
Invoking Lemma \ref{lem2} again and using the equality above, it follows that
\begin{equation*}
 T_{k+1}(x) \leq \frac{1}{\alpha}D_w(x,x_k)-\frac{1}{\alpha}D_w(x,x_{k+1})-\frac{1}{\alpha}D_w(x_{k+1},x_k)+L\cdot \ell(\tau+1)\cdot \sum_{i=1}^{\tau+1} i D_w(x_{k-\tau+i},x_{k-\tau+i-1}).
\end{equation*}
Summing the inequality over $k$ from 0 to $K$ leads to
\begin{equation}\label{sum}
\sum_{k=0}^K T_{k+1}(x) \leq \frac{1}{\alpha}D_w(x,x_0)- \frac{1}{\alpha}\sum_{k=0}^K D_w(x_{k+1},x_k)+L\cdot \ell(\tau+1)\cdot \sum_{k=0}^K \sum_{i=1}^{\tau+1} i D_w(x_{k-\tau+i},x_{k-\tau+i-1}).
\end{equation}
Moreover, it is not hard to verify that
$$\sum_{k=0}^K \sum_{i=1}^{\tau+1} i D_w(x_{k-\tau+i},x_{k-\tau+i-1})\leq  \frac{(\tau+1)(\tau+2)}{2}\sum_{k=0}^K D_w(x_{k+1},x_k).$$
This together with \eqref{sum} implies
\begin{eqnarray*}
\sum_{k=0}^K T_{k+1}(x) &\leq& \frac{1}{\alpha}D_w(x,x_0)- \left(\frac{1}{\alpha}- \frac{L\cdot \ell(\tau+1)\cdot(\tau+1)(\tau+2)}{2}\right)\sum_{k=0}^K D_w(x_{k+1},x_k)\\
        &\le& \frac{1}{\alpha}D_w(x,x_0),
\end{eqnarray*}
where the last inequality follows from the choice of $\alpha$.  Since $\{T_k(x)\}_{k\geq 0}$ is nonincreasing, it then follows that
$$T_{K+1}(x)\leq \frac{1}{\alpha \cdot(K+1)}D_w(x,x_0), ~~\forall K\geq 0.$$

\bigskip

\noindent {\bf Proof of Theorem \ref{mainresult}.}
Using the definitions of $L$ and the Bernoulli inequality, i.e., $(1+x)^r\leq 1+rx$ for any $x\geq -1$ and $r\in [0,1]$, we have
$$
\alpha_0\leq \frac{1}{\ell(\tau+1)\cdot(\tau+1)\cdot L}\leq \frac{1}{\sum_{j\in J_k}L_j} ,\quad\forall ~k\geq0.
$$
By Proposition \ref{pro1} and the choice of the step size \eqref{alpha0}, it follows that the generated sequence $\{x^k\}$ is well-defined and it  belongs to $\rint\dom w$.

We next show the convergence rate. For simplicity, we denote
$$\cY_k:= \arg\min_{x\in\cX}  D_w(x,x_k).$$
Since $\cX$ is nonempty and compact, $\cX \bigcap \dom{w} \neq \emptyset$, and $D_w(x,x_k)$ is  lsc with respect to $x$, it follows that $\cY_k$ is nonempty and compact by Weierstrass' theorem. Invoking Lemma \ref{lem2} at $x=\widetilde{x}_k\in \cY_k$ yields
\begin{equation}\label{ineq1}
\Phi(x_{k+1})\leq \Phi^*+\frac{1}{\alpha}D_w(\widetilde{x}_k,x_k)-\frac{1}{\alpha}D_w(\widetilde{x}_k,x_{k+1})-\frac{1}{\alpha}D_w(x_{k+1},x_k)+\Delta_k.
\end{equation}
Since $\widetilde{x}_k\in\cY_k\subseteq \cX$, it is easy to see that
\begin{equation}\label{ineq2}
\inf_{x\in\cX}D_w(x,x_{k+1})\leq D_w(\widetilde{x}_k,x_{k+1}).
\end{equation}
Using the Bregman distance growth condition in assumption B3 implies
\begin{equation*}
D_w(\widetilde{x}_k,x_k)=\inf_{x\in\cX}D_w(x,x_{k})\leq \frac{1}{\mu}(\Phi(x_k)-\Phi^*),
\end{equation*}
and hence
\begin{equation}\label{ineq3}
D_w(\widetilde{x}_k,x_k)\leq p D_w(\widetilde{x}_k,x_{k}) + \frac{q}{\mu}(\Phi(x_k)-\Phi^*)
\end{equation}
for any $p, q>0$ with $p+q=1$. Let
$$
p:=\frac{1}{1+\alpha\mu},~~~q:=\frac{\alpha\mu}{1+\alpha\mu}.
$$
Then it follows from \eqref{ineq1}, \eqref{ineq2}, and \eqref{ineq3} that
\begin{align*}
 &\Phi(x_{k+1})-\Phi^*+ \frac{1}{\alpha}\cdot\inf_{x\in\cX}D_w(x,x_{k+1})\\
\leq  & \frac{q}{\alpha \mu}(\Phi(x^{k})-\Phi^*)+ \frac{p}{\alpha}\cdot\inf_{x\in\cX}D_w(x,x_k)-\frac{1}{\alpha}D_w(x_{k+1},x_k)+\Delta_k\\
=& \frac{q}{\alpha \mu}\left(\Phi(x^{k})-\Phi^* + \frac{1}{\alpha}\cdot\inf_{x\in\cX}D_w(x,x_k)\right)-\frac{1}{\alpha}D_w(x_{k+1},x_k)+\Delta_k,
\end{align*}
where the last equality follows from the fact that $q = \alpha \mu p$. By the definitions of $\Gamma_\alpha(x)$ and $\Delta_k$, we immediately have
$$
\Gamma_{\alpha}(x_{k+1})\leq \frac{1}{1+\alpha\mu}\Gamma_{\alpha}(x_k)-\frac{1}{\alpha}D_w(x_{k+1},x_k)+L\cdot\ell(\tau+1)\cdot\sum_{j=k-\tau}^kD_w(x_{j+1},x_j).
$$
In order to apply Lemma \ref{lem1}, let $V_k:=\Gamma_{\alpha}(x_k)$, $w_k:=D_w(x_{k+1},x_k)$, $a:=\frac{1}{1+\alpha\mu}$, $b:=\frac{1}{\alpha }$, $c:= L\ell(\tau+1)$, and $k_0:=\tau$. When the step size $\alpha$ satisfies \eqref{alpha0}, it is not hard to verify that condition \eqref{condition} holds. Then by Lemma \ref{lem1}, the desired result \eqref{result01} follows immediately.  The result \eqref{result02} is a direct consequence of \eqref{result01}, and  the result \eqref{result03} follows from the Bregman distance growth condition and \eqref{result01}.

It remains to show \eqref{result04}. Taking $\alpha=\alpha_0$ and $Q=L/\mu$ in \eqref{result01}, and noting that
$$\frac{1}{1+\alpha_0\mu}=\left(1+\frac{1}{\ell(\tau+1)Q}\right)^{-\frac{1}{\tau+1}}=\left(1-\frac{1}{1+\ell(1+\tau)Q}\right)^{\frac{1}{\tau+1}},$$
it follows that
\begin{align}
\Gamma_{\alpha_0}(x_k)& \leq\left(\frac{1}{1+\alpha_0\mu}\right)^{k} \Gamma_{\alpha_0}(x_0) \nonumber \\
&= \left(1-\frac{1}{1+\ell(\tau+1)Q}\right)^{\frac{k}{\tau+1}} \Gamma_{\alpha_0}(x_0) \nonumber \\
&\leq \left(1-\frac{1}{[1+\ell(\tau+1)Q](\tau+1)}\right)^k \Gamma_{\alpha_0}(x_0)\label{ber},
\end{align}
where the last inequality follows from the Bernoulli inequality.

\bigskip

\noindent {\bf Proof of Lemma \ref{lem7}.}
First we observe that the following inequality holds:
\begin{eqnarray*}
\sum_{k=0}^K \frac{1}{a^{k+1}} \sum_{j=k-k_0}^k w_j &=& \frac{1}{a}(w_{-k_0} + w_{-k_0+1} + \ldots +w_0 )\nonumber \\
                                                    &&+ \frac{1}{a^2}(w_{-k_0+1} + w_{-k_0+2} + \ldots +w_1)+\ldots\nonumber \\
                                                    && +  \frac{1}{a^{k_0+1}}(w_{0} + w_{1} + \ldots +w_{k_0}) + \ldots\nonumber \\
                                                    && +\frac{1}{a^{K+1}}(w_{K-k_0} + w_{K-k_0} + \ldots +w_K)\nonumber \\
                                                    &\le & (1+\frac{1}{a}+\ldots+\frac{1}{a^{k_0}}) \sum_{k=0}^K\frac{w_k}{a^{k+1}}\nonumber \\
                                                    &=& \frac{1}{1-a}\frac{1-a^{k_0+1}}{a^{k_0}} \sum_{k=0}^K\frac{w_k}{a^{k+1}}. \label{ine-lem3}
\end{eqnarray*}
Dividing both sides of \eqref{shock2} by $a^{k+1}$ and taking the sum from $k=1$ to $k=K$  imply
$$\sum_{k=0}^K\frac{d V_{k+1}^\theta +a V_{k+1}}{a^{k+1}} \leq \left( \frac{c}{1-a}\frac{1-a^{k_0+1}}{a^{k_0}}-b\right)\sum_{k=0}^K\frac{w_k}{a^{k+1}}+\sum_{k=0}^K\frac{V_k}{a^k},\quad\forall K\geq 0.$$
It then follows from \eqref{condition} that
\begin{equation}\label{inequ-lem2}
\sum_{k=0}^K\frac{d V_{k+1}^\theta +a V_{k+1}}{a^{k+1}}\leq  \sum_{k=0}^K\frac{V_k}{a^k},\quad \forall K\geq 0.
\end{equation}
Using the convexity of $V_{k+1}^x$ in $x$ and noting that $\rho=d\theta+a$ and $a+d=1$, we have
$$V_{k+1}^\rho = V_{k+1}^{d\theta+a} \le d V_{k+1}^\theta +a V_{k+1}.$$
This together with \eqref{inequ-lem2} implies
\begin{equation}\label{inequ2-lem2}
\sum_{k=0}^K\frac{V_{k+1}^\rho}{a^{k+1}}\leq  \sum_{k=0}^K\frac{V_k}{a^k},\quad \forall K\geq 0.
\end{equation}
We next show by induction that $V_k \leq 1$ for all $k\geq 0$. First we note that $V_0\le 1$. We next show that $V_{k_0+1}\le 1$ by assuming that $V_k \leq 1$ for $k\leq k_0$. It is clear that $V_k^\rho\geq V_k$ for all $k\le k_0$ since $\rho \in (0, 1]$. Then it follows from \eqref{inequ2-lem2} and the fact that $a^0=1$ that
\begin{equation}\label{inequ3-lem2}
\frac{V_{k_0+1}^\rho}{a^{k_0+1}}+\sum_{k=1}^{k_0}\frac{V_{k} }{a^{k}}\leq \sum_{k=0}^{k_0}\frac{V_{k+1}^\rho}{a^{k+1}}\leq \sum_{k=1}^{k_0}\frac{V_k}{a^k}+V_0.
\end{equation}
Thus we have $V_{k_0+1}^\rho\leq a^{k_0+1}V_0$. Since $a<1$ and $V_0\le1$, it follows immediately that $V^\rho_{k_0+1}\le1$ and hence $V_{k_0+1}\le 1$. In conclusion, we have shown that $V_k \leq 1$ for all $k\geq 0$. Thus, it is easy to see that \eqref{inequ3-lem2} holds when $k_0$ is replaced by any nonnegative integer. We then obtain that $V_k^\rho \leq a^kV_0 $ for all $k\geq 0$.

\bigskip

\noindent {\bf Proof of Theorem \ref{Thm6}.}
Let $\cQ=\RR^d$ and $w(x)=\frac{1}{2}\|x\|^2.$ In this case, assumptions B1, B2, and B4 hold with $\ell(\cdot)$ being the identity function. By Lemma \ref{lem2}, it follows that
\begin{equation}\label{more1}
\Phi(x_{k+1})\leq \Phi(x)+\frac{1}{2\alpha}\|x-x_k\|^2-\frac{1}{2\alpha}\|x-x_{k+1}\|^2-\frac{1}{2\alpha}\|x_{k+1}-x_k\|^2+\Delta_k^1, \quad \forall x\in \RR^d,
\end{equation}
where $\Delta_k^1=\frac{(\tau+1)L}{2}\sum_{j=k-\tau}^k\|x_{j+1}-x_j\|^2.$
The H\"{o}lderian growth condition \eqref{pgrow} at $x=x_{k+1}$ reads as
\begin{equation}\label{more2}
\Phi(x_{k+1})-\Phi^*\geq \frac{\mu}{2}d^{2\theta}(x_{k+1},\cX).
\end{equation}
Let $\bar{x}_k$ be the projection point of $x_k$ onto $\cX$.
Substituting $x=\bar{x}_k$ and \eqref{more2} in \eqref{more1} yields
\begin{equation}\label{ineqad2}
\frac{\mu}{2}d^{2\theta}(x_{k+1},\cX)\leq \frac{1}{2\alpha}d^2(x_k,\cX) -\frac{1}{2\alpha}d^2(x_{k+1},\cX)-\frac{1}{2\alpha}\|x_{k+1}-x_k\|^2+\Delta_k^1.
\end{equation}
Let $V_k:=d^2(x_k,\cX)$, $w_k:=\|x_{k+1}-x_k\|^2$, $\delta:=\frac{\mu}{2} + \frac{1}{2\alpha}$, $a:=\frac{1}{2\alpha\delta}$, $b:=\frac{1}{2\alpha\delta}$, $c:=\frac{(\tau+1)L}{2 \delta}$, and $d:=\frac{\mu }{2\delta}$. Then dividing both sides of \eqref{ineqad2} by $\delta$ yields
\begin{equation*}\label{ineqad4}
d V_{k+1}^\theta + a V_{k+1}\leq a V_k -b w_k+ c\sum_{j=k-\tau}^kw_k.
\end{equation*}
It is not difficult to verify that
$a+d=1, a\in (0,1), \theta\in (0,1], b,c,d>0, V_0 \le 1.$
Moreover, when the step size $\alpha$ satisfies \eqref{alpha0}, condition \eqref{condition} holds. Then, by Lemma \ref{lem7}, we conclude that
$$d^2(x_k,\cX)\leq d^2(x_0,\cX) a^{\frac{k}{\eta}}= d^2(x_0,\cX) \left(\frac{1}{1+\alpha\mu}\right)^{\frac{k}{\eta}},\quad \forall k\geq 0,$$
where
$$\eta :=(1-a)\theta +a=(1-\frac{1}{2\alpha\delta})\theta +\frac{1}{2\alpha\delta}=\frac{\alpha\mu}{1+\alpha\mu}\theta +\frac{1}{1+\alpha\mu}.$$
When $\alpha$ is chosen as $\alpha_0$, the desired result \eqref{asym2} can be derived by using the same argument for obtaining \eqref{ber}.

\end{appendices}
\end{document}